\documentclass[12pt]{article}

\usepackage[T1]{fontenc}
\usepackage[utf8]{inputenc}
\usepackage{amsmath,amssymb,amsthm,mathtools}
\usepackage{enumitem,booktabs,microtype,graphicx,xcolor,tikz,fullpage}
\allowdisplaybreaks[4]
\newtheorem{theorem}{Theorem}
\newtheorem{lemma}{Lemma}[section]

\newtheorem{claim}{Claim}[section]

\newtheorem{conjecture}{Conjecture}

\title{Proving a conjecture concerning chromatic number, size and least eigenvalue}
\author{Leyou Xu$^a$\footnote{Email: leyouxu@m.scnu.edu.cn}, Bo Zhou$^b$\footnote{Email: zhoubo@m.scnu.edu.cn}\\
$^a$School of Computer Science, South China Normal University\\
Guangzhou 510631, P.R. China\\
$^b$School of CMathematical Sciences, South China Normal University\\
Guangzhou 510631, P.R. China
}
\date{}

\begin{document}
\maketitle

\begin{abstract}
Let $G$ be a simple nonempty graph with size $m$, chromatic number $\chi$, and least 
eigenvalue $\lambda$. We prove that 
\[
\chi(\chi-1) \le (m+1-\lambda^2)
 + \sqrt{(m+1-\lambda^2)^2-4(\lambda^2-1)(\lambda^2-m)} 
\]
with equality if and only if $G$ is either a complete graph or a complete bipartite graph, with possibly 
isolated vertices. The inequality was conjectured recently by 
Tang and Elphick in [Electron. J. Combin. 33 (2026), \#P2.65].
\\

\noindent
\textit{Keywords:} chromatic number, least eigenvalue, vertex-critical\\

\noindent
{\it 2020 Mathematics Subject Classification:} 05C50, 05C15
\end{abstract}

\section{Introduction}

Throughout this article, all graphs  are finite and simple. Let $G$ be a graph with vertex set $V(G)$ and edge set $E(G)$. The size of $G$ is $e(G)=|E(G)|$. 
A nonempty graph is a graph with size at least one. For two graphs $G$ and $H$, $G\cup H$ denotes the disjoint union of $G$ and $H$, and $G\vee H$ denotes the join of $G$ and $H$. The notations $K_n$ and $K_{a,b}$ denote the complete $n$-vertex graph and the complete bipartite graph with partite sizes $a$ and $b$. By $nK_1$ we denote the graph consisting of $n$ isolated vertices. By convention, $G\cup 0K_1=G$ for any graph $G$.
The chromatic number of a graph $G$, denoted by $\chi(G)$, is the minimum number of colors needed to color the vertices of $G$ such that no two adjacent vertices share the same color. 

Let $G$ be a graph of order $n$. 
Let $A(G)$ be the adjacency matrix of $G$. The eigenvalues of $A(G)$ are called the eigenvalues of  $G$, which are ordered as $\lambda_1(G)\ge \cdots\ge \lambda_n(G)$.  The least eigenvalue $\lambda_n(G)$ of $G$ is  denoted $\lambda (G)$. The largest eigenvalue $\lambda_1(G)$ is also known as the spectral radius of $G$. By Perron-Frobenius theorem, $-\lambda (G)\le \lambda_1(G)$.

For a graph parameter $f(G)$, when no confusion arises, we omit $G$ and simply write $f$. 

For a graph $G$ of order $n$, Tang and Elphick proved that if $3\le \chi\le n-1$, 
\begin{equation}\label{old}
\chi\le \frac{n}{2}+1+\lambda
 + \sqrt{\left(\frac{n}{2}+1+\lambda\right)^2-4(\lambda+1)\left(\lambda+\frac{n}{2}\right)}
\end{equation}
with equality if and only if $G\cong (K_{\frac{\chi}{2}}\cup \frac{n-\chi}{2}K_1)\vee (K_{\frac{\chi}{2}}\cup \frac{n-\chi}{2}K_1)$, where both $n$ and $\chi$ are even. For  $3\le \chi\le \frac{n}{2}$, this was already proved by Fan et al.~\cite{FYW} and they conjectured the range $3\le \chi\le \frac{n}{2}$ can be  extended to  $3\le \chi\le n-1$.  
Tang and Elphick observed that $\frac{n}{2}$ in \eqref{old} is closely related to the bound $\lambda\ge -\frac{n}{2}$ \cite{Con,Pow}. Let $m$ be the size of $G$. Wu and Elphick \cite{WuElphick} proved that
$\chi(\chi-1)\le (\lambda_1+1)\lambda_1\le 2m$. By analogy with \eqref{old} to replace $\frac{n}{2}$ with $m$, $\chi$ with $\chi(\chi-1)$ and $\lambda$ with $-\lambda^2$, Tang and Elphick \cite{TangElphick} proposed the following conjecture. 

\begin{conjecture} \label{con} \cite[Conjecture~8]{TangElphick}
For any nonempty graph $G$ with size $m$,
\begin{equation}\label{eq:main}
\chi(\chi-1) \le m+1-\lambda^2
 + \sqrt{(m+1-\lambda^2)^2-4(\lambda^2-1)(\lambda^2-m)}.
\end{equation}
\end{conjecture}

Tang and Elphick observed that the conjecture is immediate for bipartite graphs. They  verified it for all graphs of order at most nine and for the graphs of order at most one hundred in the Wolfram Mathematica database.
They also compared \eqref{eq:main} and \eqref{old} and found that the bound in \eqref{eq:main} typically performs better than the bound in \eqref{old}.

In this paper, we show that Conjecture \ref{con} is true and characterize the equality completely.

\begin{theorem}\label{X}
For any nonempty graph $G$ with size $m$, \eqref{eq:main} follows, and  
equality holds  if and only if $G\cong K_\chi\cup (n-\chi)K_1$ or when $\chi=2$ and $G\cong K_{a,b}\cup (n-a-b)K_1$ for some  $a,b\ge 1$, where $n$ is the order of $G$.
\end{theorem}

\section{Preliminaries}

Let $G$ be a graph. For $\emptyset\ne S\subseteq V(G)$, let $G[S]$ denote the subgraph of $G$ induced by $S$. If $S\subset V(G)$, then $G-S=G[V(G)\setminus S]$.
For $u\in V(G)$, we write $G-u$ for $G-\{u\}$. 
For $E_1\subseteq E(G)$, $G-E_1$ denotes the spanning subgraph of $G$ obtained by deleting edges of $E_1$. In particular, if $E_1=\{e\}$, we simply write $G-e$ for $G-\{e\}$.

For $v\in V(G)$, let $d_G(v)$ denote the degree of $v$ in $G$. If $S\subseteq V(G)$, we  simply write $d_S(v)$ for the number of neighbors of $v$ in $S$. 
For disjoint subsets $X,Y\subset V(G)$, we denote by $e(X,Y)$ the number of edges between vertices of $X$ and vertices of $Y$.

A clique in a graph is a set of vertices that are pairwise adjacent. The maximum size of a clique of $G$ is called the clique number of $G$, denoted by $\omega(G)$. Evidently, $\omega(G)\le \chi(G)$. 

Let $r$ be a positive integer. An $r$-chromatic graph is a graph whose chromatic number is exactly $r$. Every $r$-chromatic graph has size at least ${r\choose 2}$. 
A graph $G$ is $r$-vertex-critical if it is $r$-chromatic and  
$\chi(G-v)\le r-1$ for every $v\in V(G)$ \cite{Jen}.  
Every $r$-chromatic graph contains an induced
$r$-vertex-critical subgraph. 
Note that  for $r\ge 4$, the set of  $r$-vertex-critical graphs and the  set of $r$-critical graphs (those $G$ with $\chi (G)=r$ and  $\chi(G-t)\le r-1$ for every $t\in V(G)\cup E(G)$) are different.

%
%

\begin{lemma}\label{edge} \cite{Jen}
An $r$-vertex-critical graph has minimum degree at least $r-1$.  
\end{lemma}


\begin{lemma}\label{order}
For a noncomplete  $r$-vertex-critical graph $G$ with  $r\ge 3$, $|V(G)|\ge r+2$. 
\end{lemma}
\begin{proof}  Suppose that  $|V(G)|< r+2$. Then, as $G$ is not complete, $|V(G)|=r+1$, and there are two vertices $u$ and $v$ that  are not adjacent. By Lemma~\ref{edge}, both $u$ and $v$ are adjacent to all other vertices of $G$. Moreover, if there exists another vertex pair $\{w,z\}$ that are not adjacent, then we may color $\{w,z\}$ with one color, $\{u,v\}$ with another color and the remaining $r-3$ vertices with distinct colors. This gives an $(r-1)$-coloring, a contradiction. So $G\cong K_{r+1}-uv$, which is not $r$-vertex-critical, a contradiction. 
\end{proof}

We shall need the following structural lemma, which guarantees the existence of a clique that is critical in the sense of vertex deletion. 

\begin{lemma}\label{lem:small-extremal-coloring}
Let $G$ be an $r$-chromatic graph with $r\ge 2$ and $e(G) \le \binom{r}{2}+r-2$.
Then $G$ contains an $r$-clique $C$ such that, for every $u\in C$,
$\chi(G-u)\le r-1$.
\end{lemma}

\begin{proof}
If $r=2$, then $e(G)=1$, and $G\cong K_2\cup tK_1$ for some $t$, so the result follows. 

Assume that $r\ge 3$. 
Let $F$ be an induced $r$-vertex-critical subgraph of $G$. 
If $F\ncong K_r$, then by Lemma~\ref{order}, we have $|V(F)|\ge r+2$ and so by Lemma~\ref{edge}, $e(G)\ge e(F)\ge \frac{(r-1)(r+2)}{2}=\binom{r}{2}+r-1$, a contradiction. Hence $F\cong K_r$. Then $C=V(F)$ is an $r$-clique of $G$.

Let $u$ be an arbitrary vertex of $C$.
First, color all vertices of $C\setminus\{u\}$ with  colors $1,\dots, r-1$. Let $F_0=G-C$. As $e(G) \le \binom{r}{2}+r-2$, for each $v\in V(F_0)$, let 
\[
L(v)=\{i\in \{1,\dots, r-1\}: \text{$i$ is not used on neighbors of $v$ in $C\setminus\{u\}$}\}. 
\]
As $e(G)\le \binom{r}{2}+r-2$, there are at most $r-2$ edges of $G$ with one end outside $C$, so for each $v\in V(F_0)$,
$d_{C\setminus\{u\}}(v)+d_{F_0}(v)\le r-2$, implying that 
$|L(v)|\ge (r-1)-d_{C\setminus\{u\}}(v)\ge d_{F_0}(v)+1$.
We then color the vertices of $F_0$ greedily in any order. When a vertex $v$ is colored, at most $d_{F_0}(v)$ colors are forbidden by its already colored neighbors, so an available color remains
in $L(v)$. This gives an $(r-1)$-coloring of $G-u$, so $\chi(G-u)\le r-1$.
\end{proof}

Dirac  established a lower bound  on the size of a noncomplete $r$-vertex-critical graph for $r\ge 4$, see also \cite{BernshteynKostochka}.

\begin{theorem}\label{thm:dirac} \cite[Theorem 15]{Dirac}
Let $r\ge 4$, and let $G$ be a noncomplete $r$-vertex-critical graph. Then
$2e(G)\ge (r-1)|V(G)|+r-3$.
\end{theorem}

Given a graph $G$ that is not necessarily connected, we have by Perron-Frobenius theorem applied to each component, there is a nonnegative unit eigenvector corresponding to $\lambda_1$.
This implies the following well known lemma.

\begin{lemma}\label{subgraph}
If $F$ is a subgraph of a graph $G$, then $\lambda_1(F)\le \lambda_1(G)$.
\end{lemma}

The bound in the following lemma is known for (connected) graphs \cite{Powers}  and the equality case for connected bipartite graphs is also known \cite{HS}. However, for completeness, we include a proof here.

\begin{lemma} \label{BIS}
For a graph $G$ of size $m\ge 1$, $\lambda^2\le m$ with equality if and only if $G$ is a complete bipartite graph with possibly isolated vertices.
\end{lemma}

\begin{proof} 
Let $n=|V(G)|$. From $\sum_{i=1}^n\lambda_i^2=2m$, one has $\lambda^2+\lambda_1^2\le 2m$. As $-\lambda \le \lambda_1$, it follows that $\lambda^2\le m$ \cite{Powers}.

If $G$ is a complete bipartite graph, then it is trivial that  $\lambda^2=\lambda_1^2=m$.

Suppose that $\lambda^2=m$. Then all inequalities above are equalities, so $-\lambda=\lambda_1$ and $\lambda_i=0$ for $1<i<n$. As $G$ has exactly one positive eigenvalue, $G$ has only one non-trivial connected component, say $F$. By Perron-Frobenius theorem, $F$ is bipartite (otherwise, the index of imprimitivity of $A(F)$ is $1$, so $\lambda$ cannot be an eigenvalue of $F$ and so $G$). 
Then $F$ is a complete bipartite graph as otherwise $F$ contains a $4$-vertex induced path, so Cauchy's interlacing theorem (see \cite[Theorem 0.10]{CDS}) implies that  $F$ has two neigative eigenvalues, a contradiction.
\end{proof}

Given a nonempty graph $G$, let $\mathbf{x}$ be an eigenvector associated with $\lambda$. 
Then $\mathbf{x}$ has both positive and negative entries. Let $X = \{v : x_v \ge 0\}$ and $Y = \{v : x_v < 0\}$. Let $H$ be the spanning subgraph of $G$ with $E(H)$ to be the set of edges of $G$ between vertices $X$ and vertices of $Y$. Then $H$ is a bipartite graph with bipartition $(X, Y)$, which we call the bipartite graph of $G$ determined by $\mathbf{x}$.


\begin{lemma}\label{maxcut}
For any nonempty graph $G$ and its  bipartite graph $H$ determined by some eigenvector  $\mathbf{x}$ associated with $\lambda$, $\lambda^2 \le e(H)$.
\end{lemma}

\begin{proof}
Let $\mathbf{z}$ be a vector with  $z_v = |x_v|$ for $v\in V(G)$. From Rayleigh's principle,
\[
\lambda \mathbf{x}^\top\mathbf{x}=\mathbf{x}^\top A(G)\mathbf{x}=2\sum_{uv\in E(G)}x_ux_v\ge 2\sum_{uv\in E(H)}x_ux_v=\mathbf{x}^\top A(H)\mathbf{x}\ge \lambda(H)\mathbf{x}^\top\mathbf{x},
\]
so $\lambda\ge \lambda(H)$, implies that $\lambda^2\le \lambda^2(H)$. Now the result follows from Lemma \ref{BIS}.
\end{proof}

\section{Proof of Theorem~\ref{X}}

We are now ready to prove Theorem~\ref{X}. To prove the inequality, we first transform \eqref{eq:main} into an equivalent form, and then use a case distinction based on the relationship between the clique number and the chromatic number  and the properties of the $r$-vertex-critical graphs.

\begin{proof}[Proof of Theorem~\ref{X}]
%

As $G$ is nonempty, we have $\chi\ge 2$. Also, as $G$ is nonempty,  we have by Cauchy's interlacing theorem that $\lambda=\lambda (G)\le \lambda(K_2)= -1$, so $\lambda^2\ge 1$. By Lemma \ref{BIS}, $\lambda^2\le m$. Thus
 $(\lambda^2-1)(\lambda^2-m)\le 0$ and $m+1-\lambda^2\ge 1$. 


Suppose that $\chi=2$. Then \[
m+1-\lambda^2+\sqrt{(m+1-\lambda^2)^2-4(\lambda^2-1)(\lambda^2-m)}\ge 2(m+1-\lambda^2)\ge 2
\]
with equalities if and only if $\lambda^2=m$. So \eqref{eq:main} holds, and  by Lemma \ref{BIS}, 
\eqref{eq:main} is equality if and only if $G\cong K_{a,b}\cup (n-a-b)K_1$ with $a,b\ge 1$.

Suppose that $\chi\ge 3$. 
Note that $G$ has at least $\binom{\chi}{2}$ edges. Let $s = m-\binom{\chi}{2}$. Then $s\ge 0$. Let 
\[
f(x)=x^2 - 2(m+1-\lambda^2)x + 4(\lambda^2-1)(\lambda^2-m).
\]
Recall that  $(\lambda^2-1)(\lambda^2-m)\le 0$ and $m+1-\lambda^2\ge 1$. 
Then $f(x)=0$ has two roots $r_1$ and $r_2$ with $r_1< r_2$ and $r_1\le 0$. So when $x>0$, $x\in (r_1,r_2] \iff f(x)\le 0$ with equality if and only if $x=r_2$. 
Observe  that  \eqref{eq:main} becomes $\chi(\chi-1)\le r_2$. 
As $\chi(\chi-1)>0$, \eqref{eq:main} is equivalent to $f(\chi(\chi-1))\le 0$, i.e.,
\[
(\chi(\chi-1))^2-2(m+1-\lambda^2)\chi(\chi-1)+4(\lambda^2-1)(\lambda^2-m)\le 0.
\]
As $m=s+{\chi \choose 2}$, the above inequality is equivalent to 
\[
\lambda^4-(s+1)\lambda^2-s\left({\chi \choose 2}-1\right)\le 0,
\] 
that is, 
\[
\lambda^2 \le U_\chi(s) := \frac{s+1+\sqrt{(s-1)^2 + 2\chi(\chi-1)s}}{2}.
\]
Moreover, equality holds in \eqref{eq:main}  if and only if $\lambda^2= U_\chi(s)$. Therefore it suffices to prove that $\lambda^2 \le U_\chi(s)$ with equality if and only if $G\cong K_\chi \cup (n-\chi)K_1$.

Let $H$ be the bipartite graph of $G$ determined by some eigenvector associated with $\lambda$. Let $(X,Y)$ be its bipartition.

\noindent {\bf Case 1.}
 $\omega(G)=\chi$.

Let $C$ be a fixed $\chi$-clique of $G$.

If $s=0$, then $G\cong K_\chi\cup (n-\chi) K_1$, as desired. Suppose that $s\ge 1$.   
It suffices to show that $\lambda^2 <U_\chi(s)$.

We first show an upper bound on $-\lambda$. Let $\tau = \frac{1+\sqrt{1+4s}}{2}$. Then $\tau(\tau-1)=s$.
\begin{claim}\label{c1}
$-\lambda< \tau$. 
\end{claim}
\begin{proof}
Let $E_1=E(G[C])$. Let $M=A(G-E_1)$ and $A_0=A(G)-M$. 
Note that $A_0 + \tau I = J_\chi + D$, where $J_\chi$ is a matrix with a $\chi\times \chi$ principal submatrix of all  ones (corresponding to the vertices of $C$) and zeros elsewhere, and $D$ is a diagonal matrix with \[
D_{vv}=\begin{cases}
\tau-1, & v\in C,\\
\tau, & v\notin C.
\end{cases}
\]

We show that $D+M$ is positive semidefinite.
Let $\mathbf{x}\in  \mathbb{R}^{n}$ and $\mathbf{z}$ be a vector with $z_v=|x_v|$ for any $v\in V(G)$. Let $V^+ = \{v : x_v > 0\}$, $V^- = \{v : x_v < 0\}$ and $F$ be the spanning bipartite subgraph of $G-E(C)$ induced by the edges between $V^+$ and $V^-$.
Then \begin{align*}
\mathbf{x}^\top (D+M)\mathbf{x}&=\sum_{v\in V(G)}D_{vv}x_v^2+2\sum_{uv\in E(G-E(C))}x_ux_v\\
&\ge \sum_{v\in V(G)}D_{vv}z_v^2-2\sum_{uv\in E(F)}z_uz_v=\mathbf{z}^\top (D-A(F))\mathbf{z}.
\end{align*}
Let $W = D^{-1/2}A(F)D^{-1/2}$.
Let $\mathbf{y}$ be a nonnegative unit eigenvector corresponding to $\rho(W)$, where $\rho(W)$ is the spectral radius of $W$. Then \[
\rho(W)=2\sum_{uv\in E(F)}w_{uv}y_uy_v.
\]
By Cauchy-Schwarz inequality, we have \begin{align*}
\rho(W)^2 &\le \left(\sum_{uv\in E(F)}w_{uv}^2\right)\left( \sum_{uv\in E(F)}4y_u^2y_v^2\right)\\
&\le 4\left(\sum_{uv\in E(F)}w_{uv}^2\right)\left(\sum_{u\in V^+}y_u^2\right)\left(\sum_{u\in V^-}y_u^2\right)\le \sum_{uv\in E(F)}w_{uv}^2,
\end{align*}
where the second inequality follows as $F$ is bipartite, and the third inequality follows as $\sum_{u\in V^+}y_u^2+\sum_{u\in V^-}y_u^2\le 1$.
As every edge in $F$ either joins $C$ to $V(G)\setminus C$ or lies in $V(G)\setminus C$, we have $w_{uv}^2=\frac{1}{\tau(\tau-1)}=\frac{1}{s}$ in the former case and $w_{uv}^2=\frac{1}{\tau^2}<\frac{1}{s}$. Then $\sum_{uv\in E(F)}w_{uv}^2\le 1$. This shows that $\rho(W)\le 1$, implying that $I-W$ is positive semidefinite.
Note that $D-A(F)=D^{1/2}(I-W)D^{1/2}$. Thus $D-A(F)$ is positive semidefinite, and hence $D+M$ is positive semidefinite, as desired. 

As both  $J_{\chi}$ and $D+M$ are positive semidefinite, 
 $A(G)+\tau I=J_{\chi}+D+M$ is positive semidefinite, so  $-\lambda\le \tau$. 

Suppose that $-\lambda=\tau$. From the preceding argument, there exists a non-zero vector $\mathbf{x}$ such that $\sum_{v\in C}x_v=0$, $\mathbf{z}^\top (D-A(F))\mathbf{z}=0$ and $\rho(W)=1$, where $\mathbf{z}$ is the vector corresponding to $\mathbf{x}$ defined above. From $\rho(W)=1$, we have $F$ contains all $s$ edges outside $C$, each of these edges connects a vertex of $C$ and a vertex outside $C$. Consequently, $W=\tfrac{1}{\sqrt{s}}A(F)$ and $\lambda_1(F)^2=s=e(F)$, implying that $F$ has a non-trivial complete bipartite component with $s$ edges. So the vertices of $C$ in $F$ are all contained in either $V^+$ or $V^-$, contradicting $\sum_{v\in C}x_v=0$. So $-\lambda<\tau$.
\end{proof}

By Claim~\ref{c1},
\[
\lambda^2<\tau^2=\left(\frac{1+\sqrt{1+4s}}2\right)^2=s+\frac12+\frac12\sqrt{1+4s}.
\]
To prove $\lambda^2<U_\chi(s)$, it suffices to show that $s+\frac12+\frac12\sqrt{1+4s}\le U_\chi(s)$, or equivalently, 
$\sqrt{1+4s}\le \chi(\chi-1)-3$, and this is further equivalent to \[
s\le\frac{(\chi(\chi-1)-3)^2-1}{4}=:A_\chi.
\]
So the result holds if $s\le A_\chi$. 

Suppose next that $s>A_\chi$.
 Let $a=|X\cap C|$ and $b=|Y\cap C|$. Then $a+b=\chi$ and 
\[
e(G[X])+e(G[Y])\ge {a\choose 2}+{b\choose 2}=a^2-\chi a+\frac{1}{2}(\chi^2-\chi)
\ge \left\lfloor\frac{(\chi-1)^2}{4}\right\rfloor.
\]
So 
Lemma~\ref{maxcut} gives
$$\lambda^2\le e(H)=m-(e(G[X])+e(G[Y]))\le m-\left\lfloor\frac{(\chi-1)^2}{4}\right\rfloor=s+\binom{\chi}{2}-\left\lfloor\frac{(\chi-1)^2}{4}\right\rfloor.$$
To prove $\lambda^2 <U_\chi(s)$, it suffices to show that
 $s+\binom{\chi}{2}-\left\lfloor\frac{(\chi-1)^2}{4}\right\rfloor< U_\chi(s)$, or equivalently, 
\[
s> \frac{\left(\chi(\chi-1)-2\left\lfloor\frac{(\chi-1)^2}{4}\right\rfloor\right)\left(\chi(\chi-1)-2\left\lfloor\frac{(\chi-1)^2}{4}\right\rfloor-2\right)}{4\left\lfloor\frac{(\chi-1)^2}{4}\right\rfloor}=:R_\chi.
\]
 If $\chi=2h\ge4$, then
\[
A_\chi-R_\chi=2(2h^2-1)(h^2-h-1)\ge0,
\]
and if $\chi=2h+1\ge3$, then
\[
A_\chi-R_\chi=
\frac{(h-1)(h+1)(4h^3+4h^2-2h-1)}h\ge0.
\]
Thus  $s>A_\chi\ge  R_\chi$, as desired.

\noindent {\bf Case 2.}  $\omega(G)<\chi$. 

It suffices to show that $\lambda^2 <U_\chi(s)$.

\noindent {\bf Case 2.1.}  $\chi=3$. 

Note that  $G$ is triangle free and non-bipartite. As $H\ne G$, Lemma~\ref{maxcut} gives $\lambda^2\le m-1=s+2$. It suffices to show  $s+2<U_3(s)$, i.e., 
\[
s+2<\frac{1}{2}\left(s+1+\sqrt{(s-1)^2+12s} \right),
\] 
 or equivalently, $s> 2$. In fact, a triangle-free non-bipartite graph contains an odd cycle of length at least five, so $m\ge5$ and
$s=m-3$. If $m\ge 6$, then $s>2$, as desired. If $m=5$, then 
$G\cong C_5\cup (n-5)K_1$, and by a direct calculation, $\lambda=-\frac{1+\sqrt5}{2}$, and hence
$\lambda^2=\frac{3+\sqrt5}{2}<4=U_3(2)$. 

\noindent {\bf Case 2.2.}  $\chi\ge4$. 

Let
$\beta=m-e(H)$. 
Let $B_\chi=\left\lfloor\frac{(\chi-1)^2}{4}\right\rfloor
 +\left\lfloor\frac{\chi}{2}\right\rfloor-1$.

\begin{claim}\label{lem:maxcut-deficiency}
 $\beta\ge B_\chi$.
\end{claim}

\begin{proof}
Let $a=\chi(G[X])$ and $b=\chi(G[Y])$. 
As $G[X]$ and $G[Y]$ can be colored with disjoint sets of colors, $a+b\ge \chi$. Then
\[
\beta\ge\binom a2+\binom b2.
\]

Suppose to the contrary that $\beta<B_\chi$. As $\beta$ is an integer,
\[
\beta\le \left\lfloor\frac{(\chi-1)^2}{4}\right\rfloor
 +\left\lfloor\frac{\chi}{2}\right\rfloor-2\le \begin{cases}
h^2-2 &\text{if }\chi=2h,\\
h^2+h-2 &\text{if }\chi=2h+1.
\end{cases}
\]
If $a+b\ge \chi+1$, then
\[
\binom a2+\binom b2\ge
\begin{cases}
h^2 &\text{if }\chi=2h,\\
h^2+h &\text{if }\chi=2h+1,
\end{cases}
\]
a contradiction. So $a+b=\chi$.

Assume that $a\le b$. 
If $\chi=2h$, write $a=h-r$ and $b=h+r$, where $r\ge0$. Then
$\binom a2+\binom b2=\left\lfloor\frac{(\chi-1)^2}{4}\right\rfloor+r^2$, 
and hence
\[
\beta-\binom a2-\binom b2
 \le h-2-r^2\le h-r-2=a-2.
\]
If $\chi=2h+1$, write $a=h-r$ and $b=h+1+r$. Then
$\binom a2+\binom b2=\left\lfloor\frac{(\chi-1)^2}{4}\right\rfloor+r(r+1)$,
and hence
\[
\beta-\binom a2-\binom b2
 \le h-2-r(r+1)\le h-r-2=a-2.
\]
So in both cases, $\beta-\binom a2-\binom b2\le a-2$.
It then follows that $a\ge2$. 
Note that $\beta=e(G[X])+e(G[Y])$. Then by Lemma~\ref{edge}, $e(G[Y])\ge \binom b2$, we have  \[
e(G[X])\le \binom a2+\binom b2+a-2-e(G[Y])\le \binom a2+a-2.
\]
Similarly, $e(G[Y])\le \binom b2+a-2\le \binom b2 +b-2$. 
By Lemma~\ref{lem:small-extremal-coloring},
$G[X]$ contains an $a$-clique $A$ and $G[Y]$ contains a $b$-clique $B$ such that deleting any vertex of the corresponding clique lowers the chromatic number  by at least one.

Since $|A|+|B|=a+b=\chi$ and $G$ contains no $K_\chi$, there are vertices $u\in A$ and $v\in B$ with
$uv\notin E(G)$. Color $G[X]-u$ with $a-1$ colors and $G[Y]-v$ with a disjoint set of $b-1$ colors,
and give $u$ and $v$ one new common color. This is a proper $(\chi-1)$-coloring of $G$, a contradiction.
Thus $\beta\ge B_\chi$. 
\end{proof}

Let $d_\chi=\binom{\chi}{2}-B_\chi$ and $L_\chi=\frac{d_\chi(d_\chi-1)}{B_\chi}$.

\begin{claim}\label{lem:critical-estimates}
There is an integer $a\ge2$ such that $s\ge\frac{(\chi-1)a+\chi-3}{2}$ and $\lambda_1\ge \chi-1+\frac{\chi-3}{\chi+a}$.
\end{claim}

\begin{proof}
Let $F$ be an induced $\chi$-vertex-critical subgraph of $G$. As $\omega(G)<\chi$, $F$ is not complete. By Lemma~\ref{order}, $|V(F)|\ge \chi+2$. Let $a=|V(F)|-\chi$. By  the fact that $e(F)\le m=\binom \chi2+s$ and Theorem~\ref{thm:dirac}, $2(\binom \chi2+s)\ge 2e(F)\ge(\chi-1)(\chi+a)+\chi-3$, so the first inequality follows. 
Moreover, by Lemma~\ref{subgraph} and Rayleigh's principle, we have
$\lambda_1(G)\ge\lambda_1(F)\ge\frac{2e(F)}{\chi+a}\ge \chi-1+\frac{\chi-3}{\chi+a}$,
which proves the second inequality.
\end{proof}

\noindent {\bf Case 2.2.1} $s\ge L_\chi$.

As $s\ge L_\chi$, we have $4B_\chi s-4d_\chi(d_\chi-1)\ge 0$, equivalently, 
\[
s+2d_\chi-1\le\sqrt{(s-1)^2+2\chi(\chi-1)s},
\] 
so $s+d_\chi\le U_\chi(s)$. It then follows by Lemma~\ref{maxcut} and Claim \ref{lem:maxcut-deficiency} that 
$\lambda^2\le m-\beta\le m-B_\chi=s+d_\chi\le U_\chi(s)$.

Suppose that $\lambda^2 = U_\chi(s)$. 
Then $s+d_\chi\le U_\chi(s)$, so  $s=L_\chi$.
If $\chi=2h$, then $B_\chi=h^2-1$, $d_\chi=h^2-h+1$ and $L_\chi=\frac{h(h^2-h+1)}{h+1}$. As $s$ is an integer,  $(h+1)|h(h^2-h+1)$. As $h(h^2-h+1)\equiv -3\pmod {h+1}$, we have $(h+1)|3$, which implies that $h=2$. So $\chi=4$ and $s=L_\chi=2$. However, we have by Claim~\ref{lem:critical-estimates}, $s\ge \frac{3a+1}{2}\ge \frac{7}{2}$, a contradiction. 
If $\chi=2h+1$, then 
$B_\chi=h^2+h-1$, $d_\chi=h^2+1$, and $L_\chi=\frac{h^2(h^2+1)}{h^2+h-1}$.
As $L_\chi$ is an integer, we have $(h^2+h-1)|h^2(h^2+1)$. Moreover, 
$h^2(h^2+1)\equiv3-4h\pmod{h^2+h-1}$. If $h\ge 3$, then
$(h^2+h-1)-(4h-3)=(h-1)(h-2)>0$, a contradiction. So $h=2$, $\chi=5$ and $s=L_\chi=4$.
By Claim~\ref{lem:critical-estimates} again, we have $s\ge \frac{4a+2}{2}\ge 5$, also a contradiction. So $\lambda^2< U_\chi(s)$.

\noindent {\bf Case 2.2.2.} $s< L_\chi$.

Suppose first that $\chi=4,5,6$. By a direct calculation, $L_4=2$, $L_5=4$ and $L_6=\frac{21}{4}$. However, Claim~\ref{lem:critical-estimates} gives $s\ge \frac{3\chi-5}{2}$, contradicting $s<L_\chi$. So $\chi\ge 7$.

Suppose that $\chi=7$. By Claim~\ref{lem:critical-estimates}, $s\ge8$. As $s$ is an integer and $s<L_7<9$, we have $s=8$. Again by Claim~\ref{lem:critical-estimates},
$8\ge\frac{6a+4}{2}=3a+2$, so $a=2$. Moreover, $\lambda_1(G)\ge6+\frac49=\frac{58}{9}$. So $\lambda^2\le2\left(\binom72+8\right)-\left(\frac{58}{9}\right)^2
 =\frac{1334}{81}<17$. On the other hand, $U_7(8)=\frac{9+\sqrt{721}}2>17$. Thus, $\lambda^2<U_7(8)$, as desired.

Suppose in the following that $\chi\ge8$.
If $\chi=2h$, then
$L_\chi=\frac{h(h^2-h+1)}{h+1}<h^2=\frac{\chi^2}{4}$.
If $\chi=2h+1$, then $L_\chi=\frac{h^2(h^2+1)}{h^2+h-1}$ and $\frac{\chi^2}{4}=h^2+h+\frac14$. As 
$$\left(h^2+h+\frac14\right)(h^2+h-1)-h^2(h^2+1)
 =\frac{8h^3-3h^2-3h-1}{4}>0,$$ we have $L_\chi<\frac{\chi^2}{4}$. So in both cases,  $L_\chi<\frac{\chi^2}{4}$. 
If $a\ge \chi$, then Claim~\ref{lem:critical-estimates} gives
$s\ge\frac{(\chi-1)\chi+\chi-3}{2}=\frac{\chi^2-3}{2}>\frac{\chi^2}{4}$,
contradicting $s<L_\chi<\frac{\chi^2}{4}$. Hence $a<\chi$, so $\chi+a\le2\chi-1$. It then follows by Claim~\ref{lem:critical-estimates} that
$\lambda_1(G)\ge \chi-1+\frac{\chi-3}{2\chi-1}\ge \chi-\frac23$. 
So 
\[
\lambda^2\le 2m-\lambda_1^2\le \chi(\chi-1)+2s-\left(\chi-\frac23\right)^2
 =2s+\frac{\chi}{3}-\frac49. 
\]
We are left  to prove $2s+\frac{\chi}{3}-\frac49<U_\chi(s)$ whenever
$\frac{3\chi-5}{2}\le s<L_\chi$, which is equivalent to $3s+\frac{2\chi}{3}-\frac{17}{9}<\sqrt{(s-1)^2+2\chi(\chi-1)s}$. Since $s\ge \frac{3\chi-5}{2}$ and $\chi\ge 8$, we have $3s+\frac{2\chi}{3}-\frac{17}{9}\ge \frac{93\chi-169}{18}\ge 0$, so the above inequality is further equivalent to \[
324s^2-(81\chi^2-243\chi+378)s+18\chi^2-102\chi+104=:P_\chi(s)<0. 
\]
Note that
\begin{align*}
P_\chi\left(\frac{3\chi-5}{2}\right)&=-\frac{243\chi^3-2628\chi^2+7413\chi-6148}{2}\\
&=-\frac{243(\chi-8)^3+3204(\chi-8)^2+12021(\chi-8)+9380}{2}<0.
\end{align*} 
Moreover, if $\chi=2h$, then $h\ge 4$ and
\begin{align*}
P_{\chi}(L_\chi) &=-\frac{2(3h^2-11h+13)(27h^3-12h^2+11h-4)}{(h+1)^2}\\
&=-\frac{2\left(3(h-4)^2+13(h-4)+17\right)\left(27(h-4)^3+312(h-4)^2+1211(h-4)+1576\right)}{(h+1)^2}\\
&<0,
\end{align*}
and if $\chi=2h+1$, then $h\ge 4$, and  
\begin{align*}
P_\chi(L_\chi)&=-\frac{2(3h^3-14h^2+16h-10)(27h^4+15h^3+14h^2-7h+1)}{(h^2+h-1)^2}\\
&=-\frac{2\left(3(h-4)^3+22(h-4)^2+48(h-4)+22\right)}{(h^2+h-1)^2}\\
&\quad \cdot \left( 27(h-4)^4+447(h-4)^3+2786(h-4)^2+7737(h-4)+8069\right)\\
&<0.
\end{align*} 
As $P_\chi$ is convex in $s$, we have $P_\chi(s)<0$ whenever $\frac{3\chi-5}{2}\le s<L_\chi$. Therefore, $\lambda^2<U_\chi(s)$. 

Combining all cases, we complete the proof.
\end{proof}

\bigskip
\noindent {\bf Acknowledgements.} 
This work was supported by National Natural Science Foundation of China (No.~12571364).

\end{document}